\documentclass[12pt, reqno]{amsart}

\usepackage{amsmath, amsthm, amssymb, enumerate, amscd}

\theoremstyle{plain}
\newtheorem{theorem}{Theorem}[section]
\newtheorem{proposition}[theorem]{Proposition}
\newtheorem{corollary}[theorem]{Corollary}

\theoremstyle{definition}
\newtheorem{definition}[theorem]{Definition}

\newtheorem{remark}[theorem]{Remark}
\newtheorem{question}[theorem]{Question}
\newtheorem{example}[theorem]{Example}
\newtheorem{examples}[theorem]{Examples}

\newcommand{\N}{{\mathbb{N}}}
\newcommand{\Z}{{\mathbb{Z}}}
\newcommand{\R}{{\mathbb{R}}}
\newcommand{\F}{{\mathbb{F}}}
\newcommand{\p}{{\mathfrak{p}}}
\newcommand{\q}{{\mathfrak{q}}}
\newcommand{\m}{{\mathfrak{m}}}
\newcommand{\afrak}{{\mathfrak{a}}}
\newcommand{\EC}{{\mathcal{E}}}

\begin{document}

\title[Pseudomeadows] {Some properties of pseudomeadows}

\author{Hamid Kulosman}
\address{Department of Mathematics\\ 
University of Louisville\\
Louisville, KY 40292, USA}
\email{hamid.kulosman@louisville.edu}



\date{}

\begin{abstract} 
The purpose of this paper is to study the commutative pseudomeadows, the structure which is defined in the same way as commutative meadows, except that the existence of a multiplicative identity is not required. We extend the characterization of finite commutative meadows, given by I.~Bethke, P.~Rodenburg, and A.~Sevenster in their 2015 paper, to the case of commutative pseudomeadows with finitely many idempotents. We also extend the well-known characterization of general commutative meadows as the subdirect products of fields to the case of commutative pseudomeadows. Finally we investigate localizations of commutative pseudomeadows.
\end{abstract}

\maketitle

%---------------------------------
\section{Introduction}

The notion of a {\it regular ring} (in the context of non-commutative rings) was introduced by J.~von Neumann in 1930 in \cite{n} for the purpose of clarifying some concepts that appeared in two of his papers in the area of Functional Analysis, that were published at that time. Since in Commutative Algebra the notion of a regular ring has a different meaning, the name widely used from 1960's on is {\it von Neumann regular rings} (in both commutative and non-commutative context). Another name often used in Commutative Algebra is {\it absolutely flat rings}. From the point of view of Homological Algebra they were studied in detail in the paper \cite{o}, where the definition (in the context of Homological Algebra) was attributed to N.~Bourbaki (see \cite[Ch. 1 \S 2 ex. 16 and 17, ch. 2 \S 4 ex. 16]{b}). Finally, in recent years the notion of a {\it meadow} appeared in the literature (see, for example, the papers \cite{bb, brs}). Since I am exclusively interested in commutative multiplication, I consider the notion of a commutative  meadow to be the same as the notion of a commutative von Neumann regular ring (in accordance with \cite[Definition 2.7]{brs}). The precise relation between these two notions from the point of view of Logic is explained, for example, in \cite{bb}.

\medskip
The inspiration for writing this paper is coming from the paper \cite{brs} by I.~Bethke, P.~Rodenburg, and A.~Sevenster (for which I wrote a review \cite{kul} for Mathematical Reviews). 
In \cite{brs} the authors characterized finite commutative meadows as direct products of fields. I was wondering if their characterization could be extended to the context of pseudorings. (Pseudorings are defined in the same way as rings, the only difference is that the existence of a multiplicative identity is not required.) I gave the name {\it pseudomeadows} to the pseudoring version of meadows. My interest was in the commutative pseudomeadows only and I was actually able to show that commutative pseudomeadows  with finitely many idempotents are finite direct products of fields, thus generalizing the characterization from \cite{brs}. The proof is in the spirit of \cite{brs}. This is done in Section 3.

\medskip
It was then natural to try to characterize  general commutative pseudo\-mea\-dows. In Section 4 I first extended to commutative pseudomeadows the well-known characterization of commutative meadows as reduced rings in which every prime ideal is maximal. One of the difficulties in dealing with pseudorings instead of rings is the fact that not every pseudoring has maximal ideals, and, even when they do have them, they are not necessarily prime (even when the ring in question is reduced). That is probably the reason why pseudomeadows were not studied in the literature. (The only exception, to the best of my knowledge, is an exercise in Kaplansky's book \cite{k}, which provided one of the directions of the just mentioned characterization.) The next step was the cha\-rac\-te\-rization of commutative pseudomeadows as subdirect products of fields. The meadow version of that statement was given by G.~Birkhoff in \cite{b} and is attributed to G.~K\"othe (see \cite{ko}), but the arguments in these sources needed to be completed in order to amount to a proof of the statement.

\medskip
Finally in Section 5 I investigated localizations of pseudomeadows and showed that for each maximal ideal $\m$ of a pseudomeadow $R$ we have that $R_\m$ is a field which is naturally isomorphic to the field $R/\m$, and in particular that $\m_\m=(0)$. There are in the literature meadow versions of some of the statements, but the proofs are different (as I had to deal with the non-existence of the identity element, which makes everything more complicated). At the end of the section I proved two local-global principles for commutative pseudomeadows, which are standard for commutative rings, but, in general, do not hold in commutative pseudorings.

\medskip
In each od the sections 2,3,4,5 I obtained some ``simpler" properties of commutative pseudomeadows and von Neumann invertible elements, not all of which were used in the ``major" statements, but are interes\-ting in their own right.  

%----------------------------------
\section{Notation and preliminaries}

\begin{definition}
A {\it commutative pseudoring} $(R,0,+,\cdot,-)$  is a set $R$ with a distinguished element $0$ (called the {\it zero element}), two binary operations $(x,y)\mapsto x+y: R\times R\to R$ and $(x,y)\mapsto x\cdot y: R\times R\to R$ (denoted also $(x,y)\mapsto xy$) and an unary operation $x\mapsto -x:R\to R$, such that the following axioms hold:

(1) $(x+y)+z=x+(y+z)$ for all $x,y,z\in R$;

(2) $x+y=y+x$ for all $x,y\in R$;

(3) $x+0=x$ for all $x\in R$;

(4) $x+(-x)=0$ for all $x\in R$;

(5) $(xy)z=x(yz)$ for all $x,y,z\in R$;

(6) $xy=yx$ for all $x,y\in R$;

(7) $x(y+z)=xy+xz$ for all $x,y,z\in R$;

(8) $(x+y)z=xz+yz$ for all $x,y,z\in R$. 
\end{definition}

From now on we will call commutative pseudorings just {\it pseudorings}.

\begin{definition}
Let $R$ be a pseudoring and $x$ an element of $R$. We say that $x$ is {\it von Neumann invertible} if there exists an element $x^{(-1)}\in R$ such that $xxx^{(-1)}=x$ and $xx^{(-1)}x^{(-1)}=x^{(-1)}$. We call the element $x^{(-1)}$ a {\it von Neumann inverse} of $x$.
\end{definition}

\begin{proposition}[{\cite[Lemma 2]{o}}]\label{vNi_first}
Let $R$ be a pseudoring. An element $x\in R$ is von Neumann invertible if and only if $x\in Rx^2$.
\end{proposition}

\begin{proof}
If $x$ is von Neumann invertible, then $x=xxx^{(-1)}$, hence $x\in Rx^2$. Conversely, suppose that $x=rx^2$ for some $r\in R$. Then one can check that the element $x^{(-1)}=r^2x$ is a von Neumann inverse of $x$. 
\end{proof}

The proof of the next simple proposition, given in \cite{o}, works for rings, but not for pseudorings.

\begin{proposition}[{\cite[Lemma 2]{o}}]
Let $R$ be a pseudoring and $x\in R$. If there exists a von Neumann inverse of $x$, it is unique.
\end{proposition}

\begin{proof}
Suppose that $x^{(-1)}$ and $x^{((-1))}$ are both von Neumann inverses of $x$. Then $x^{(-1)}x^2x^{((-1))}=(x^{(-1)}x^2)x^{((-1))}=xx^{((-1))}$, but also 
$x^{(-1)}x^2x^{((-1))}$  $=x^{(-1)}(x^2x^{((-1))})=x^{(-1)}x$. Hence $xx^{(-1)}=xx^{((-1))}$. If we multiply this equality once by $x^{(-1)}$ and second time by $x^{((-1))}$ we get $x^{(-1)}xx^{((-1))}=x^{(-1)}$ and $x^{(-1)}xx^{((-1))}=x^{((-1))}$ respectively. Thus $x^{(-1)}=x^{((-1))}$.
\end{proof}

\begin{definition}
A {\it pseudomeadow} $(R,0,+,\cdot,-,^{(-1)})$  is a pseudoring $(R,0,+,\cdot,-)$ which has an additional unary operation $x\mapsto x^{(-1)}:R\to R$, such that, in addition to the axioms of a pseudoring, the following axiom holds:

(9) $xxx^{(-1)}=x$ and $xx^{(-1)}x^{(-1)}=x^{(-1)}$ for all $x\in R$.
\end{definition}

\begin{definition}
A {\it ring} (resp. {\it meadow}) $(R,0,1,+,\cdot,-)$ (resp. $(R,0,1,+,$  $\cdot,-,^{(-1)})$) is a pseudoring $(R,0,+,\cdot,-)$ (resp. pseudomeadow $(R,0,+,\cdot,$ $-,^{(-1)})$) with an additional distinguished element $1$ (called the {\it identity element}) such that the following axiom holds in addition to the axioms (1)-(8) (resp. (1)-(9)):

(10) $x\cdot 1=x$ for all $x\in R$.

\noindent When in a ring or meadow $R$ for an $x\in R$ there exists an $x'\in R$ such that $xx'=1$, such an $x'$ is unique and is called the {\it multiplicative inverse} of $x$. It is denoted by $x^{-1}$. We then say that $x$ is {\it invertible}.
\end{definition}

For every invertible element $x$ of a meadow $R$ we have $x^{(-1)}=x^{-1}$. Also every invertible element $x$ of a ring $R$ is von Neumann invertible and $x^{(-1)}=x^{-1}$. 

\smallskip
If $(R,0,+,\cdot,-)$ (resp. $(R,0,1,+,\cdot,-)$) is a pseudoring (resp. ring) such that there exists a unary operation $x\mapsto x^{(-1)}:R\to R$ which makes it a pseudomeadow (resp. meadow), we sometimes say that a pseudoring (resp. ring) $R$ is a pseudomeadow (resp. meadow).

\begin{examples} (1) Every field is meadow. Indeed, on every field $(F,0,1,+,\cdot,-)$ we have a unary operation $x\mapsto x^{(-1)}:F\to F$, defined by
\begin{equation*}
x^{(-1)}=
    \begin{cases}
      0, & \text{if $x=0$,}\\
      x^{-1}, & \text{if $x\ne 0$,}
     \end{cases}
\end{equation*}
which makes it a meadow. A meadow of this type is called a {\it zero-totalized field}.

(2) Every product of pseudomeadows (resp. meadows) is a pseudomeadow (resp. meadow). In particular, every product of fields is a meadow. (Indeed, we have a unary operation $(x_i)\mapsto (x_i)^{(-1)}:=(x_i^{(-1)})$ which makes $\prod R_i$ a pseudomeadow (resp. meadow).)
\end{examples}

\begin{definition}
A subpseudoring of the direct product $\prod_{i\in I} R_i$  of a family of pseudorings is called a {\it subdirect product} of the family $(R_i)_{i\in I}$ if for every $i_0\in I$ and every $y\in R_{i_0}$ there is an element $(x_i)\in R$ such that $x_{i_0}=y$.
\end{definition}

\begin{definition}
Let $R$ be a pseudoring. An element $e\in R$ is called an {\it idempotent} of $R$ if $e^2=e$. The set of all idempotents of $R$ is denoted by $\EC(R)$. (It is non-empty as $0\in\EC(R)$.) On the set $\EC(R)$ we introduce a partial order in the following way: for any $e,f\in\EC(R)$ we put $e\le f$ if $ef=e$. (It is easy to verify that this, indeed, is a partial order.) An element $e\in\EC(R)$ is called a {\it minimal idempotent} of $R$ if $e\ne 0$ and for every non-zero $f\in\EC(R)$, $f\le e$ implies $f=e$. Two idempotents $e,f$ are said to be {\it orthogonal} if $ef=0$. For every von Neumann invertible element $x\in R$ we define $e(x):=xx^{(-1)}$.
\end{definition}

The statements from the next proposition are proved in \cite{brs} for the case of rings, but everything works for pseudorings without any change. We will often be using these statements, usually without explicitly referring to the proposition.

\begin{proposition}[{\cite{brs}}]\label{idem_tools}
Let $R$ be a pseudoring.

(a) Every idempotent $e\in R$ is von Neumann invertible and $e^{(-1)}=e$. 

(b) If $e,f$ are two idempotents of $R$, then $ef$ is an idempotent of $R$.

(c) Any two minimal idempotents of $R$ are orthogonal.

(d) If $e,f$ are two idempotents of $R$ and $e\le f$, then $f-e$ is also an idempotent.

(e) If $e,f$ are two orthogonal idempotents of $R$, then $e+f$ is an idempotent.

(f) If $x\in R$ is von Neumann invertible, then $e(x):=xx^{(-1)}$ is an idempotent. Moreover, $e(x)=0$ if and only if $x=0$. 

(g) If $R$ is a pseudomeadow and $e$ an idempotent of $R$, then $Re$ is a meadow with the identity element $e$.

(h) If $R$ is a pseudomeadow and $e$ a minimal idempotent of $R$, then $Re$ is a field with the identity element $e$. Moreover, if $xe\ne 0$, then $(xe)^{-1}=x^{(-1)}e$.
\end{proposition}

\begin{definition}
Let $R,R'$ be pseudorings. A map $h:R\to R'$ is called a {\it pseudoring homomorphism} if $h(x+y)=h(x)+h(y)$ and $h(xy)=h(x)h(y)$ for all $x,y\in R$. If $R,R'$ are rings, then $h$ is a {\it ring homomorphism} if it, additionally, satisfies $f(1)=1$. A bijective pseudoring homomorphism (resp. ring homomorphism) is called a {\it pseudoring isomorphism} (resp. {\it ring isomorphism}). We will call pseudoring homomorphisms (resp. pseudoring isomorphisms) simply {\it homomorphisms} (resp. {\it isomorphisms}).
\end{definition}

The next proposition is easy to prove.

\begin{proposition}\label{idem_preservation}
(a) If $R,R'$ are two pseudorings and $h:R\to R'$ an isomorphism, then $h$ induces a restriction $h:\EC(R)\to \EC(R')$ and it is an isomorphism of partially ordered sets. In particular, $e$ is a minimal idempotent of $R$ if and only if $h(e)$ is a minimal idempotent of $R'$.

(b) Every isomorphism $h:R\to  R'$ between two rings is necessarily a ring isomorphism.
\end{proposition}

\begin{definition}
Let $R$ be a pseudoring.

(a) An ideal $\p$ of $R$ is said to be {\it prime} if it is proper and
\begin{equation*}
(\forall\,x,y\in R)\; xy\in\p \Rightarrow x\in \p \text{ or } y\in \p.
\end{equation*} 
The set of all prime ideals of $R$ is called the {\it spectrum} of $R$ and denoted by $\mathrm{Spec}(R)$. The intersection of all prime ideals of $R$ is called the {\it nil radical} of $R$ and is denoted by $\mathfrak{N}(R)$.

(b) An ideal $\m$ of R is said to be {\it maximal} if it is proper and if for every ideal $I$ of $R$, $I\supseteq \m$ implies $I=\m$ \text{ or } $I=R$.
The set of all maximal ideals of $R$ is called the {\it maximal spectrum} of $R$ and denoted by $\mathrm{MaxSpec}(R)$. The intersection of all maximal ideals of $R$ is called the {\it Jacobson radical} of $R$ and is denoted by $\mathfrak{J}(R)$.

(c) $R$ is called a {\it pseudodomain} if $R\ne \{0\}$ and
\begin{equation*}
(\forall\,x,y\in R)\; xy=0 \Rightarrow x=0 \text{ or } y=0.
\end{equation*}

(d) An element $x\in R$ is said to be {\it nilpotent} if $x^n=0$ for some $n\ge 1$.

(e) $R$ is  said to be {\it reduced} if it is without non-zero nilpotent elements.
\end{definition}

The next proposition is easy to prove. 

\begin{proposition}\label{quotient_pseudodomain}
Let $R$ be a pseudoring and $\p$ an ideal of $R$. Then $\p$ is prime if and only if $R/\p$ is a pseudodomain.
\end{proposition}

\begin{remark}
Let $R$ be a pseudoring and $\m$ a maximal ideal of $R$. Then $\m$ is not necessarily prime and $R/\m$ is not necessarily a pseudodomain, in particular, not necessarily a field. 
\end{remark}

\begin{example}
Let $R$ be the additive group $\Z_6$ with the all-products-zero multipliction. Then $\m_1=\{0,2,4\}$ and $\m_2=\{0,3\}$ are all the maximal ideals of $R$. None of them  is prime.
(For example, $1\cdot 1=0\in\m_1$, but $1\notin\m_1$.) Hence $R/\m_1$ is not a pseudodomain, in particular, not a field. Same for $R/\m_2$. We also have $\mathfrak{J}(R)=(0)$. As no ideal of $R$ is prime, we have $\mathfrak{N}(R)=R$. 

Similar things can happen even if $R$ is reduced. For example, in the subpseudoring $2\Z$ of $\Z$ the ideal $\m=4\Z$ is the only maximal ideal, but it is not prime (as $2\cdot 2=4\in\m$, but $2\notin\m$). In this case the pseudoring $R/\m$ has two elements and its underlying group is isomorphic to the group $\Z_2$, however it is equipped with the all-products-zero multiplication.
\end{example}

\begin{proposition}\label{quotient_field}
Let $R$ be a pseudoring and $\m$ a maximal ideal of $R$. If $\m$ is prime, then $R/\m$ is a field.
\end{proposition}

\begin{proof}
$R/\m$ is a non-zero pseudoring whose only ideals are $(0)$ and $R/\m$. Let $a\in R\setminus \m$. Then the ideal $R/\m.(a+\m)$ of $R/\m$ is $\ne (0)$ as $(a+\m)(a+\m)=a^2+\m\ne 0$ (since $a\notin \m$ and $\m$ is prime by assumption). Hence $R/\m.(a+\m)=R/\m$. Hence there exists an $e+\m\in R/\m$ such that $(e+\m)(a+\m)=a+\m$. We claim that $e+\m$ is an identity element of $R/\m$. Indeed, suppose that for some $x,y\in R$ we have $(e+\m)(x+\m)=y+\m$. Hence $(e+\m)(a+\m)(x+\m)=(a+\m)(y+\m)$, whence 
$(a+\m)(x-y+\m)=0$. As $R/\m$ is a pseudodomain (by Proposition \ref{quotient_pseudodomain}) and $a+\m\ne 0$, we have $x+\m=y+\m$. Thus $e+\m$ is an identity element of $R/\m$. Let $s+\m$ be any non-zero element of $R/\m$. Since $R/\m.(s+\m)=R/\m$, there is an $t+\m\in R/\m$ such that $(s+\m)(t+\m)=e+\m$. Hence $s+\m$ is invertible and so $R/\m$ is a field. 
\end{proof}

The next proposition is the slightly extended \cite[Theorem 2]{m}.

\begin{proposition}[{\cite[Theorem 2]{m}}]\label{charact_red}
Let $R$ be a pseudoring. The nil radical $\mathfrak{N}(R)$ consists precisely of all nilpotent elements of $R$.
In particular, $R$ is reduced if and only if $\mathfrak{N}(R)=(0)$.
\end{proposition}

\begin{proposition}[{\cite[page 492]{m}}]\label{no_nilp}
Let $R$ be a pseudoring and $x\in R\setminus \{0\}$. If $x$ is von Neumann invertible, it is not nilpotent.
\end{proposition}
 
\begin{proof}
Suppose $x$ is nilpotent and let $n\ge 2$ be the smallest number such that $x^n=0$. Then $x^{n-2}x^2x^{(-1)}=0$, hence $x^{n-1}=0$, a contradiction.
\end{proof}

A nonempty subset $S$ of a pseudoring $R$ is called a {\it multiplicatively closed set} (or said to be {\it multiplicatively closed}) if $s_1, s_2\in S$ implies $s_1s_2\in S$.
The localization $S^{-1}R$ is defined as for the rings. $S^{-1}R$ is a ring as the element $s/s$ $(s\in S)$ is the identity element. As usual,  $S^{-1}R$ is denoted $R_\p$ when $S=R\setminus \p$ for a prime ideal $\p$ of $R$. If $I$ is an ideal of $R$, we define
\begin{equation*}
I_\p=\{a/s\;:\; a\in I, s\in R\setminus\p\}.
\end{equation*}
$I_\p$ is an ideal of $R_\p$, called the {\it localization of the ideal $I$ at $\p$}. We have $(I+J)_\p=I_\p+J_\p$. If $I$ is contained in $\p$, the $\p/I$ is a prime ideal of $R/I$ and the canonical map $f:(R/I)_{\p/I}\to R_\p/I_\p$, given by $\displaystyle{f(\frac{r+I}{s+I})=\frac{r}{s}+I_\p}$, is an isomorphism. If $J\supseteq I$, then under this map $f$ we have $\displaystyle{f((J/I)_{\p/I})=J_\p/I_\p}$. Note also that $(R/I)\setminus (\p/I)=\{s+I\;:\;s\in R\setminus \p\}$. The prime ideals of $R_\p$ are precisely the ideals $\q_\p$, where $\q$ is a prime ideal of $R$ contained in $\p$, and if $\q\ne \q'$, then $\q_\p\ne\q'_\p$. 

%-----------------------------------
\section{Pseudomeadows with finitely many idempotents}

\begin{example}\label{examples_with_inf_min_idem}
{\it A pseudomeadow with infinitely many minimal idempotents can be with, as well as without, an identity ellement.}

Indeed, in the pseudoring $\F_2^{(\N)}$, consisting of all sequences $x=(x_1, x_2, x_3,\dots)$ of elements of $\F_2$ with finite support, every element is an idempotent, so that this ring is a pseudomeadow. For each $n\ge 1$ let $e(n)$ be the sequence whose $n$-th component is $1$ and all other components are $0$. The elements $e(n)$, $n\in\N$, are precisely all the minimal idempotents of $\F_2^{(\N)}$. This pseudomeadow does not have an identity element.

 Similarly, in the ring $\F_2^\N$, consisting of all sequences $x=(x_1, x_2, x_3,\dots)$ of elements of $\F_2$, every element is an idempotent, so this ring is a meadow. The elements $e(n)$, $n\in\N$, are again precisely all the minimal idempotents (but in this case there is an identity element).
\end{example}

\begin{example}\label{meadow_no_min_idem}
{\it There is a non-zero meadow without minimal idempotents.}

Indeed, let $E=[0,1)$ with addition $x+y=\max\{x,y\}$. Then $E$ is a monoid. Consider the monoid ring $M=\F_2[X;E]$. We assume that each element of $M$ is written in the form $f=X^{\alpha_1}+\dots+X^{\alpha_n}$ with $n\ge 0$ and $0\le \alpha_1<\dots <\alpha_n<1$. The addition and multiplication in $M$ are defined in the standard way (analogously to the operations with polynomials). We first claim that every element $f\in M$ is an idempotent and we prove this by induction on the number $n$ of terms of $f$. The cases $n=0$ and $n=1$ are obvious. Suppose the claim is true for $n$ terms . Let $f=X^{\alpha_1}+\dots+X^{\alpha_n}+X^{\alpha_{n+1}}$. Then (using the inductive hypothesis and the characteristic $2$ of $M$) $f^2=(X^{\alpha_1}+\dots+X^{\alpha_n})(X^{\alpha_1}+\dots+X^{\alpha_n})+X^{\alpha_{n+1}}(X^{\alpha_1}+\dots+X^{\alpha_n})+(X^{\alpha_1}+\dots+X^{\alpha_n})X^{\alpha_{n+1}}+X^{\alpha_{n+1}}=X^{\alpha_1}+\dots+X^{\alpha_n}+X^{\alpha_{n+1}}=f$. The claim is proved. In particular, $M$ is a meadow.

We now prove that $M$ has no minimal idempotents. Let $f=X^{\alpha_1}+\dots+X^{\alpha_n}\ne 0$. Consider two cases. \underbar{First case: $n$ odd.} Let $\beta>\alpha_n$. Then $g=X^\beta$ is a smaller idempotent than $f$ as $fg=(X^{\alpha_1}+\dots+X^{\alpha_n})X^\beta=X^\beta=g$. \underbar{Second case: $n\ne 0$ even.} Let $\beta$ be such that $\alpha_{n-1}<\beta<\alpha_n$ and let $g=X^\beta+X^{\alpha_n}$. Then $g=X^\beta$ is a smaller idempotent than $f$ as 
\begin{align*}
fg&=(X^{\alpha_1}+\dots+X^{\alpha_n})(X^\beta+X^{\alpha_n})\\
   &=(\underbrace{X^\beta+\dots+X^\beta}_{n-1}+X^{\alpha_n})+(\underbrace{X^{\alpha_n}+\dots+X^{\alpha_n}}_{n})\\
   &=X^\beta+X^{\alpha_n}\\
   &=g.
\end{align*}
\end{example}

\begin{proposition}\label{1_sum_some_min_idem}
Let $M$ be a meadow in which $1=e_1+\dots+e_n$ for some minimal idempotents $e_1, \dots, e_n$ of $M$. Then $e_1, \dots, e_n$ are all minimal idempotents of $M$ and they are pairwise distinct. Moreover, $M$ has finitely many idempotents.
\end{proposition}
\begin{proof}
Suppose there is a minimal idempotent $e_{n+1}$ different than $e_1, \dots, e_n$. Then $e_{n+1}=e_{n+1}(e_1+\dots+e_n)$, hence (since minimal idempotents are pairwise orthogonal) $e_{n+1}=0$, a contradiction. Thus $e_1, \dots, e_n$ are all minimal idempotents of $M$. They are pairwise distinct, otherwise if $e_i$ appears $t\ge 2$ times in $1=e_1+\dots+e_n$, we would have (after multiplying both sides of this equality by $e_i$) $e_i=e_i+\dots+e_i$ ($t$ times), hence $0=e_i+\dots+e_i$ ($t-1$ times). Let $r\ge 2$ be the smallest number such that $0=e_i+\dots+e_i$ ($r$ times). Then $1=1-(\underbrace{e_i+\dots+e_i}_{r})$, hence (after multiplying each side by $e_i$) $e_i=e_i-(\underbrace{e_i+\dots+e_i}_{r})$, hence $0=e_i+\dots+e_i$ ($r-1$ times), a contradiction.

 Let now $e$ be any idempotent of $M$. Since $ee_i\cdot e_i=ee_i$, we have that the idempotent $ee_i$ is $\le e_i$ for every $i=1,\dots, n$.  Hence for every $i=1,\dots, n$, $ee_i$ is either $0$ or $e_i$ (as the $e_i$ are minimal idempotents). Now
\begin{align*}
e&=e\cdot 1\\
  &=e(e_1+\dots+e_n)\\
  &=ee_1+\dots+ee_n\\
  &=e_{i_1}+\dots+e_{i_k}
\end{align*}
for some distinct $i_1, i_2, \dots,i_k$ from $\{1,\dots,n\}$. Here $k\ge 0$. (If $k=0$, then $e=0$, otherwise $e\ne 0$.)
Since each idempotent of $M$ has this form, $M$ has finitely many idempotents.
\end{proof}

The next theorem is a generalization of \cite[Lemma 3.6 and Theorem 3.7]{brs} to the case of pseudomeadows, which, moreover, are not necessarily finite.

\begin{theorem}\label{pseudomeadow_finitely_many_min_idem}
Let $M$ be a pseudomeadow with finitely many minimal idempotents and suppose that for every non-zero idempotent $e\in M$ there exists a mini\-mal idempotent $e'\in M$ such that $e'\le e$. Then:

(a) $M$ is a meadow whose identity element is equal to the sum of all minimal idempotents of $M$.

(b) Moreover, $M$ is isomorphic to a finite product of fields, namely $M\cong Me_1\times \cdots\times Me_n$, where $\{e_1, \dots, e_n\}$ is the set of all minimal idempotents of $M$.

(c) $M$ has finitely many idempotents and each of them is a sum of distinct minimal idempotents.
\end{theorem}
\begin{proof}
(a) Let $M$ be a pseudomeadow satisfying the conditions from the statement. If $M=\{0\}$, the statement is true, so we assume from now on that $M\ne \{0\}$. If $x$ is any non-zero element of $M$, then $xx^{(-1)}$ is an idempotent of $M$, different than $0$. Hence $M$ contains at least one non-zero idempotent. Since by assumption there is a minimal idempotent $e'$ such that $e'\le xx^{(-1)}$, $M$ has at least one minimal idempotent. Let $e_1, \dots, e_n$ ($n\ge 1$) be all distinct minimal idempotents of $M$. We will show that $e_1+\dots+e_n$ is an identity element of $M$. 

Let $e$ be any idempotent of $M$. Since minimal idempotents are pairwise orthogonal, $e_1+\dots+e_n$ is an idempotent of $M$, hence $(e_1+\dots+e_n)e$ is an idempotent of $M$ and $(e_1+\dots+e_n)e\le e$. Hence $e-(e_1+\dots+e_n)e$ is an idempotent of $M$. Suppose $e-(e_1+\dots+e_n)e\ne 0$. By the assumption there is a minimal idempotent, say $e_i$, such that $e_i\le e-(e_1+\dots+e_n)e$. Hence $e_i[e-(e_1+\dots+e_n)e]=e_i$. But the left-hand side of this equality is equal to $0$, hence $e_i=0$, a contradiction. Thus 
\begin{equation}\label{eq1}
e=(e_1+\dots+e_n)e
\end{equation}
for every idempotent $e$.

Let $x$ be any element of $M$. Then $xx^{(-1)}$ is an idempotent of $M$, so that, by (\ref{eq1}), $xx^{(-1)}(e_1+\dots+e_n)=xx^{(-1)}$. Hence $xxx^{(-1)}(e_1+\dots+e_n)=xxx^{(-1)}$, i.e., $x(e_1+\dots+e_n)=x$. Thus $e_1+\dots+e_n$ is an identity element of $M$. 

(b) By Proposition \ref{idem_tools} each $Me_i$ is a field with the identity element $e_i$, so that the element $(e_1,\dots, e_n)$ is the identity element of the meadow $Me_1\times \cdots\times Me_n$. We define a map $h:M\to Me_1\times\cdots\times Me_n$ by
\begin{equation*}
h(x)=(xe_1,\dots, xe_n) \text{ for every $x\in M$}.
\end{equation*}
This map is a ring homorphism since it is easy to see that it is a pseudoring homorphism and since it satisfies $h(e_1+\dots+e_n)=((e_1+\dots+e_n)e_1,\dots, (e_1+\dots+e_n)e_n)=(e_1,\dots, e_n)$. It is injective since $h(x)=h(y)$ implies $xe_1=ye_1$,\dots, $xe_n=ye_n$, hence $x(e_1+\dots+e_n)=y(e_1+\dots+y_n)$, hence $x=y$. Finally, it is surjective since for any $(x_1e_1,\dots, x_ne_n)\in Me_1\times\cdots\times Me_n$ there is an element $x\in M$, namely $x=x_1e_1+\dots+x_ne_n$, such that $h(x)=(x_1e_1,\dots, x_ne_n)$. Thus $h$ is an isomorphism of the meadows $M$ and $Me_1\times\cdots\times Me_n$.

(c) The minimal idempotents in $Me_1\times\cdots\times Me_n$ are $(e_1,0,\dots,0), (0,e_2,$ $0,\dots,0),\dots, (0,0,\dots, e_n)$ and every non-zero idempotent of $Me_1\times\cdots\times Me_n$ is a sum of distinct minimal idempotents . Hence by Proposition \ref{idem_preservation} and the part (b), every non-zero idempotent of $M$ is a sum of distinct minimal idempotents of $M$.
\end{proof}

Every infinite field is an infinite meadow with finitely many idempotents. There the sum of minimal idempotents is equal to $1$. However, we have the following example.

\begin{example}\label{meadow_sum_min_idem_not_1}
{\it There is a meadow which has a finite non-zero number of minimal idempotents, in which the sum of all minimal idempotents is not equal to $1$.}

Indeed, let $E=[0,1]$ with addition $x+y=\max\{x,y\}$. Then $E$ is a monoid. Consider the monoid ring $M=\F_2[X;E]$. We assume that each element of $M$ is written in the form $f=X^{\alpha_1}+\dots+X^{\alpha_n}$ with $n\ge 0$ and $0\le \alpha_1<\dots <\alpha_n\le1$. The addition and multiplication in $M$ are defined in the standard way (analogously to the operations with polynomials). Similarly as in Example \ref{meadow_no_min_idem} one can show that all the elements of $M$ are idempotents. In particular, $M$ is a meadow. We claim that $X^1$  is the only minimal idempotent of $M$. Let us first show the minimality. We have:
\[(X^{\alpha_1}+\dots+X^{\alpha_n})\cdot X^1=nX^1=\begin{cases} X^1, & \text{if $n$ is odd,}\\ 0, & \text{if $n$ is even.}\end{cases}\]
Hence $X^1$ is a minimal idempotent. Now we show that no other element of $M$ is a minimal idempotent. Suppose to the contrary. Let $f=X^{\alpha_1}+\dots+X^{\alpha_n}$ be a minimal idempotent and $f\ne X^1$. If $n$ is odd, then $f\cdot X^1=X^1$, a contradiction. If $n\ge 2$ is even, then for $\alpha_{n-1}<\beta<\alpha_n$ we have $f\cdot (X^\beta+X^{\alpha_n})=(n-1)X^\beta+X....^{\alpha_n}+nX^{\alpha_n}=X^\beta+X^{\alpha_n}$, a contradiction. Thus $X^1$ is the only minimal idempotent of $M$ and it is not equal to $1$. Our argument also proves that for any idempotent $f=X^{\alpha_1}+\dots+X^{\alpha_n}$ with $n\ge 2$ even there is no minimal idempotent $f'$ such that $f'\le f$.
\end{example}

\begin{corollary}\label{pseudomeadow_finitely_many_idem}
Let $M$ be a pseudomeadow with finitely many idempotents.

(a) $M$ has an identity element. It is equal to the sum of all minimal idempotents of $M$.

(b) If $\{e_1,\dots, e_n\}$ is the set of all minimal  idempotents of $M$, then $M\cong Me_1\times\dots\times Me_n$.
\end{corollary}

\begin{proof}
This immediately follows from the previous theorem.
\end{proof}

The statement analogous to this corollary, but for pseudorings, is not true, as we can see in the next simple example.

\begin{example}\label{pseudoring_finitely_many_idem}
{\it There is a pseudoring with finitely many idempotents, which does not have an identity element.}

Indeed, the simplest example is $R=2\Z$ with the operations induced from $\Z$. If we want to have at least one non-zero idempotent (so that we have at least one minimal idempotent), we can take $R=2\Z\times \Z$, or $R=2\Z\times \Z\times \Z$ (to have all kinds of idempotents). In $R=2\Z\times\Z\times\Z$ we have two minimal idempotents, namely $(0,1,0)$ and $(0,0,1)$, and two idempotents that are not minimal: $(0,0,0)$ and $(0,1,1)$. The sum of minimal idempotents is $(0,1,1)$, which is not an identity element of $R$. In particular, $R$ is not a pseudomeadow (by Corollary \ref{pseudomeadow_finitely_many_idem}, but it is also easy to check directly). Hence Theorem \ref{pseudomeadow_finitely_many_min_idem} cannot be applied even though for each non-zero idempotent $e$ there is a minimal idempotent $e'$ such that $e'\le e$.
\end{example}

\begin{corollary}[{\cite[Lemma 3.6 and Theorem 3.7]{brs}}]\label{BRS_main}
Let $M$ be a finite meadow and $\{e_1,\dots, e_n\}$ the set of all minimal  idempotents of $M$. Then:

(a) $e_1+\dots+e_n=1$;

(b) $M\cong Me_1\times\dots\times Me_n$.
\end{corollary}

\begin{proof}
This immediately follows from the previous corollary.
\end{proof}

%---------------------------------------------
\section{Characterizations of general pseudomeadows}

\begin{proposition}\label{pseudoring_over_p}
Let $R$ be a pseudoring and  $\p$ be a prime ideal of $R$. If there exists a von Neumann invertible element $x\in R\setminus \p$, then $R/\p$ is a domain and $e(x)+\p$ is its identity element.
\end{proposition}

\begin{proof}
$R/\p$ is a pseudodomain by Proposition \ref{quotient_pseudodomain}. Let $r+\p\in R/\p$ and suppose that $(e(x)+\p)(r+\p)=s+\p$ for some $s\in R$.Then $(x+\p)(xx^{(-1)}+\p)(r+\p)=(x+\p)(s+\p)$, hence $(x+\p)(r+\p)=(x+\p)(s+\p)$. Since $R/\p$ is a pseudodomain and $x+\p\ne 0$, we have $s+\p=r+\p$. Thus $e(x)+\p$ is an identity element of $R/\p$ and $R/\p$ is a domain.
\end{proof}

\begin{corollary}\label{e-e}
Let $R$ be a pseudoring and $\p$ a prime ideal of $R$. Let $x,y$ be two von Neumann invertible elements of $R$ not contained in $\p$. Then $e(x)-e(y)\in\p$.
\end{corollary}

\begin{proof}
Each of $e(x)+\p$ and $e(y)+\p$ is the identity element of $R/\p$ by the previous proposition. Hence $e(x)+\p=e(y)+\p$.
\end{proof}

\begin{proposition}
Let $R$ be a pseudoring and $x\in R\setminus \{0\}$ a von Neumann invertible element. Then $x$ is not contained in the Jacobson radical of $R$.
\end{proposition}

\begin{proof}
The element $x$ is not nilpotent by Proposition \ref{no_nilp}. Hence $S_x=\{x, x^2, x^3,\dots\}$ is a multiplicatively closed set which does not contain $0$. Let $\q$ be an ideal of $R$ maximal with respect to the property $\q\cap S_x=\emptyset$. (It exists by Zorn's lemma.) By a standard argument one can show that $\q$ is prime. Hence by Proposition \ref{pseudoring_over_p}, $R/\q$ is a domain and $e(x)+\q$ is its identity element. Suppose that there is a proper ideal $\afrak$ of $R$ which contains $\q$. Then $x^i\in\afrak$ for some $i$. Hence $e(x)^i\in\afrak$, so that $(e(x)+\q)^i\in \afrak/\q$. This means that the identity element $e(x)+\q$ of $R/\q$ is an $\afrak/\q$, whence $\afrak/\q=R/\q$. Hence $\afrak=R$, a contradiction. Thus $\q$ is a maximal ideal of $R$. Since $x\notin \q$, $x$ does not belong to the Jacobson radical of $R$.
\end{proof}

\begin{proposition}\label{psedomeadow_pseudodomain}
Every pseudomeadow $R$ which is a pseudodomain is a field. For any $x,y\in R\setminus \{0\}$, $e(x)=e(y)$, and the element $e(x)$ $(x\ne 0)$ is the identity element. The inverse of any $x\ne 0$ is the element $x^{(-1)}$.
\end{proposition}

\begin{proof}
Let us first assume that $R$ is a meadow which is a domain. Let $x\in R\setminus \{0\}$. We have $xxx^{(-1)}-x=0$, hence $x(xx^{(-1)}-1)=0$, hence $xx^{(-1)}=1$. Thus $x$ is invertible and so $R$ is a field.

Suppose now that $R$ is a pseudomeadow which is a pseudodomain.  Let $x\in R\setminus \{0\}$. Let $y\in R$. Then $xxx^{(-1)}-xy=0$, hence $x(xx^{(-1)}y-y)=0$, hence $xx^{(-1)}y=y$. Thus $e(x)$ is an identity element of $R$. (In particular, $e(x)=e(y)$ for any $x,y\ne 0$.) Hence $R$ is a meadow which is a domain, whence (by the first part of the proof) $R$ is a field. As $xx^{(-1)}$ is the identity element (for any $x\ne 0$), we have that $x^{(-1)}$ is the inverse of $x$.
\end{proof}

\begin{corollary}\label{pseudomeadow_over_p}
Let $R$ be a pseudomeadow and $\p$ a prime ideal of $R$. Then $R/\p$ is a field. The identity element is $e(x)+\p$ for any $x\notin \p$. The inverse of an $x+\p$ $(x\notin \p)$ is the element $x^{(-1)}+\p$.
\end{corollary}

\begin{proof}
$R/\p$ is a pseudomeadow, which is (by Proposition \ref{quotient_pseudodomain}) a pseudodomain. Hence by the previous proposition, $R/\p$ is a field. The identity element (by the previous proposition) is the element $(x+ \p)(x+\p)^{(-1)}=(x+\p)(x^{(-1)}+\p)=e(x)+\p$ for any $x\notin\p$ and the inverse of $x+\p$ $(x\notin\p)$ is $x^{(-1)}+\p$.
\end{proof}

The next theorem is a characterization of pseudomeadows. The part (b) $\Rightarrow$ (a) is stated by Kaplansky in \cite[page 64, Exercise 22]{k}. The proof of that part that we give here follows his recommendation. In the case of rings the theorem is stated and proved (in a different way) by Goodearl in \cite[Theorem 1.16]{g}.   

\begin{theorem}\label{charact_pseudomeadow}
Let $R$ be a pseudoring. The following are equivalent:

(a) $R$ is a pseudomeadow;

(b) $R$ is reduced and $\mathrm{Spec}(R)=\mathrm{MaxSpec}(R)$.
\end{theorem}

\begin{proof}
\underbar{(a) $\Rightarrow$ (b):} Let $R$ be a pseudomeadow. By Proposition \ref{no_nilp}, $R$ is reduced. 

Let $\p$ be a prime ideal of $R$. Suppose that $\afrak$ is a proper ideal of $R$ which contains $\p$. Then $\afrak/\p$ is a proper, non-zero ideal of $R/\p$. This contradicts to the fact that $R/\p$ is a field by Corollary \ref{pseudomeadow_over_p}. Hence $\mathrm{Spec}(R)\subseteq \mathrm{MaxSpec}(R)$.

Let $\m$ be a maximal ideal of $R$. Suppose that $\m$ is not prime. Let $a,b\in R$ be such that $a,b\notin \m$, but $ab\in\m$. Let $x$ be an element of $R\setminus \m$. Then $\m\cap \{x,x^2,x^3,\dots\}=\emptyset$. (Suppose to the contrary. Let $n\ge 2$ be the smallest number such that $x^n\in\m$. Then $x^{n-2}x^2x^{(-1)}\in\m$, i.e., $x^{n-1}\in\m$, a contradiction.) Since $\m$ is maximal, $\m+Ra$ contains some $x^i$, i.e., $m_1+r_1a=x^i$ for some $m_1\in\m$, $r_1\in R$, and $i\ge 1$. Similarly $m_2+r_2a=x^j$ for some $m_2\in\m$, $r_2\in R$, and $j\ge 1$. Multiplying these two relations we conclude that $x^{i+j}\in\m$, a contradiction. Hence $\m$ is prime and so $\mathrm{MaxSpec}(R)\subseteq \mathrm{Spec}(R)$. Thus $\mathrm{MaxSpec}(R)= \mathrm{Spec}(R)$.

\smallskip
\underbar{(b) $\Rightarrow$ (a):} Suppose (b) holds. Let $a\in R$. Let 
\begin{equation*}
S=\{a^n-a^{n+1}y\;:\;n\ge 1, y\in R\}.
\end{equation*}
The set $S$ is multiplicatively closed. Suppose $0\notin S$. Let $\p$ be an ideal of $R$ maximal with respect to the property $\p\cap S=\emptyset$. By a standard argument one shows that the ideal $\p$ is prime, hence maximal (as  $\mathrm{Spec}(R)=\mathrm{MaxSpec}(R)$). By Proposition \ref{quotient_field}, $R/\p$ is a field.  We have $a\notin \p$ (otherwise $a^n-a^{n+1}y\in\p$ for any $n$, which is not true). Hence $a^2\notin\p$ and so $R/\p.(a^2+\p)=R/\p$. Hence $(r+\p)(a^2+\p)=a+\p$ for some $r\in R$, i.e., $a-ra^2\in\p$, a contradiction. Thus the hypothesis that $0\notin S$ is wrong. Let $n\ge 1$ and $y\in R$ be such that $a^n-a^{n+1}y=0$. If $n=1$, we have $a-a^2y=0$, hence $a$ is von Neumann invertible by Proposition \ref{vNi_first}. Suppose $n\ge 2$. We have the identity:
\begin{align*}
(a-a^2y)^n &=(a^n-a^{n+1}y)- (n-1)ay(a^n-a^{n+1}y)+\\
  &\binom{n-1}{2}a^2y^2(a^n-a^{n+1}y)-\dots+(-1)^{n-1}a^{n-1}y^{n-1}(a^n-a^{n+1}y).
\end{align*}
Hence $(a-a^2y)^n=0$. Since $R$ is reduced, $a-a^2y=0$. Hence $a$ is von Neumann invertible by Proposition \ref{vNi_first}. Thus $R$ is a pseudomeadow.
\end{proof}

The next proposition is an extension of \cite[Theorem 3]{m}.

\begin{proposition}\label{emb_red_pseudoring}
Let $R$ be a pseudoring. The following are equivalent:

(a) $R$ is reduced;

(b) $R$ can be embedded in a pseudomeadow;

(c) $R$ can be embedded in a direct product of fields.
\end{proposition}

\begin{proof}
\underbar{(a) $\Rightarrow$ (c):} This is \cite[Theorem 3]{m}.

\underbar{(c) $\Rightarrow$ (b):} A direct product of fields is a meadow, hence (c) $\Rightarrow$ (b).

\underbar{(b) $\Rightarrow$ (a):} Let $R$ be a pseudoring which is a subpseudoring of a pseudomeadow $R'$. Let $a\in R\setminus \{0\}$. By Theorem \ref{charact_pseudomeadow} there is a prime ideal $\p'$ of $R'$ such that $a\notin\p'$. Then $\p=\p'\cap R$ is a prime ideal of $R$ such that $a\notin\p$. Hence $\mathfrak{N}(R)=(0)$, i.e., $R$ is reduced (by Proposition \ref{charact_red}).
\end{proof}

\begin{remark}
The statement of the next theorem is given in \cite[Corollary 4]{b} and is attributed to \cite[page 552]{ko}. However, the author in \cite{ko} does not deal with von Neumann regular rings, so that the arguments in both \cite{b} and \cite{ko} need to be completed in order to amount to a proof of the next theorem.
\end{remark}

\begin{theorem}[{\cite{b, ko}}]\label{subdirect_fields}
Any pseudomeadow is isomorphic to a subdirect product of fields.
\end{theorem}

\begin{proof}
By Proposition \ref{quotient_field} and Theorem \ref{charact_pseudomeadow}, for every maximal ideal $\m$ of $R$ the quotient $R/\m$ is a field. Also by Theorem \ref{charact_pseudomeadow} and Proposition \ref{charact_red}, $\cap_{\m\in\mathrm{MaxSpec}(R)} \m=(0)$. Hence the canonical map $f:R\to \prod_{\m\in\mathrm{MaxSpec}(R)} R/\m$, given by $f(x)=(x+\m)_{\m\in\mathrm{MaxSpec}(R)}$, is injective. It is clear that $f(R)$ is a subdirect product of the family $(R/\m)_{\m\in\mathrm{MaxSpec}(R)}$.
\end{proof}

In connection with the previous theorem, it would be interesting to answer the following question.

\begin{question}
Characterize the pseudorings that are isomorphic to a subdirect product of fields.
\end{question}

Note that besides the pseudomeadows (by Theorem \ref{subdirect_fields}), the ring $\Z$ has that property. Any pseudoring which has that property must be reduced. The zero ring, $\{0\}$, does not have that property.
%----------------------------------------------------------------
\section{Localization of pseudomeadows}

\begin{proposition}\label{vNi_elt_in_ideal}
Let $R$ be a pseudoring, $I$ an ideal of $R$, and $x$ a von Neumann invertible element of $R$. Then $x\in I$ if and only if $x^{(-1)}\in I$.
\end{proposition}

\begin{proof}
If $x\in I$ then $x^{(-1)}=xx^{(-1)}x^{(-1)}\in I$. By symmetry, if $x^{(-1)}\in I$ then $x\in I$.
\end{proof}

\begin{theorem}\label{R_m field}
Let $R$ be a pseudomeadow and $\m$ a maximal ideal of $R$. Then: 

(a) $R_\m$ is a field. In particular, $\m_\m=(0)$.

(b) The canonical map $f_\m:R\to R_\m$, given by $\displaystyle{f_\m(x)=\frac{xs}{s}}$ ($s$ any element of $ R\setminus \m$), is surjective. Its kernel is $\m$, so that, by passing to the quotient, the map $\widetilde{f_\m}:R/\m\to R_\m$, given by $\displaystyle{\widetilde{f_\m}(x+\m)=\frac{xs}{s}}$, is an isomorphism.

(c) $f_\m^{-1}(s/s)=e(s)+\m$, where $s$ is any element of $R\setminus \m$. 

(d) $e(s_1)-e(s_2)\in\m$ for any $s_1, s_2\in R\setminus \m$.
\end{theorem}

\begin{proof}
(a) Let $x\in R$ and $s\in R\setminus \m$. Then (by Proposition \ref{vNi_elt_in_ideal}) $s^{(-1)}\in R\setminus \m$. Now $\displaystyle{\frac{x}{s}\frac{x}{s}\frac{x^{(-1)}}{s^{(-1)}}=\frac{x}{s}}$ and $\displaystyle{\frac{x}{s}\frac{x^{(-1)}}{s^{(-1)}}\frac{x^{(-1)}}{s^{(-1)}}=\frac{x^{(-1)}}{s^{(-1)}}}$, so $R_\m$ is a pseudomeadow. The identity element of $R_\m$ is $\displaystyle{\frac{s}{s}}$, $s\in R\setminus \m$. It follows from Theorem \ref{charact_pseudomeadow} that $\mathrm{MaxSpec}(R_\m)=\mathrm{Spec}(R_\m)=\{\m_\m\}$. Since (again by Theorem \ref{charact_pseudomeadow}) $R_\m$ is reduced, then by Proposition \ref{charact_red}, $\m_\m=(0)$. Hence the non-zero elements of $R_\m$ are precisely the elements $r/s$ with both $r,s$ from $R\setminus \m$. Each of them is invertible as $\displaystyle{\frac{r}{s}\frac{s}{r}=\frac{rs}{rs}}$. Thus $R_\m$ is a field.

(b) The canonical map $f_\m:R\to R_\m$ is a pseudoring homomorphism. It is surjective since for any $x/s\in R_\m$ we have $\displaystyle{f_\m(xs^{(-1)})=\frac{xs^{(-1)}s}{s}=\frac{x}{s}}$ as $(xs^{(-1)}s^2-xs)s'=0$ for any $s'\in R\setminus \m$. Also $\mathrm{Ker}(f_\m)=\{x\in R\;:\;xs=0 \text{ for some }s\in R\setminus\m\}\subseteq \m$. Since $\mathrm{Ker}(f_\m)=f_\m^{-1}(0)$ and $(0)$ is a prime ideal of $R_\m$, we have that $\mathrm{Ker}(f_\m)$ is a prime ideal of $R$. Hence by Theorem \ref{charact_pseudomeadow}, $\mathrm{Ker}(f_\m)=\m$. Thus  the map $\widetilde{f_\m}:R/\m\to R_\m$, given by $\displaystyle{\widetilde{f_\m}(x+\m)=\frac{xs}{s}}$, is an isomorphism.

(c) We have $\displaystyle{f_\m(ss^{(-1)})=\frac{ss^{(-1)}s}{s}=\frac{s}{s}}$. Hence $f_\m^{-1}(s/s)=ss^{(-1)}+\m$.

(d) Since $s_1/s_1=s_2/s_2$ for any $s_1,s_2\in R\setminus\m$,
we have by (c) that $s_1s_1^{(-1)}+\m=s_2s_2^{(-1)}+\m$, i.e., $e(s_1)-e(s_2)\in\m$. (Alternatively, the relation $e(s_1)-e(s_2)\in\m$ follows from Corollary \ref{e-e}.)
\end{proof}

\begin{example}
Let $R$ be the ring $\Z_6$ and $\m=\{0,2,4\}$, a maximal ideal of $R$. (One can easily check that $R$ is a meadow. This also follows from \cite[Theorem 2.11]{brs}.) We have  $R\setminus \m=\{1,3,5\}$. By the previous theorem, $|R_\m|=|R/\m|=|R|/|\m|=2$, so that $R_\m\cong \F_2$. Note that $1\cdot 1^{(-1)}-3\cdot 3^{(-1)}=4\in\m$, $1\cdot 1^{(-1)}-5\cdot 5^{(-1)}=2\in\m$, $3\cdot 3^{(-1)}-5\cdot 5^{(-1)}=4\in\m$. Note also that here the set $f_\m^{-1}(1/1)=1+\m=\{1,3,5\}$ strictly contains the set $\{e(s)\;:\;s\in R\setminus\m\}=\{1,3\}$. 
\end{example}

\begin{example}
A pseudoring $R$ is called a {\it Boolean pseudoring} if $x^2=x$ for every $x\in R$. (It then follows that $2x=0$ for every $x\in R$.) Let $R$ be a Boolean pseudoring and $\m$ a maximal ideal of $R$. $R$ is a pseudomeadow and $x^{(-1)}=x$ for every $x\in R$. Fix an $x\in R\setminus \m$. Let $y$ be any element from $R\setminus \m$. By the part (d) of the previous theorem, or by Corollary \ref{e-e}, we have $xx^{(-1)}-yy^{(-1)}\in \m$, hence $x-y\in\m$. Hence there are only two cosets in $R$ with respect to $\m$, so that $R/\m\cong \F_2$. By the part (b) of the previous theorem we have $\R_\m\cong \F_2$. Also, by Theorem \ref{subdirect_fields}, we have that $R$ is isomorphic to a subdirect product of a family of the fields $\F_2$.

Let us now describe the maximal ideals of the pseudomeadow $R=\F_2^{(\N)}$, which consists of all sequences of elements of $\F_2$ with only finitely many non-zero coordinates. For $n\ge 1$ let $e(n)$ be the element of $R$ which has all coordinates equal to $0$ except the $n$-th coordinate, which is $1$. Let $\m$ be a maximal ideal of $R$. Suppose that for an $n\ge 1$ there is an element $x=(x_1, x_2, x_3,\dots)\in \m$ with $x_n=1$. Then $e(n)=e(n)\cdot x\in\m$. Since not all $e(n)$ $(n\ge 1)$ are in $\m$, there is an $r\ge 1$ such that all elements of $\m$ have the $r$-th coordinate equal to $0$. Since $\m$ is a maximal ideal, we conclude that $\m$ consists precisely of all the elements of $R$ which have the $r$-th coordinate equal to $0$. Denote this maximal ideal by $\m(r)$. We deduce (using Theorem \ref{charact_pseudomeadow}) that 
\begin{equation*}
\mathrm{MaxSpec}(R)=\mathrm{Spec}(R)=\{\m(r)\;:\;r\ge 1\}.
\end{equation*}
\end{example}

\medskip
At the end we prove two local-global principles for pseudomeadows, that are standard for rings (see \cite[page 79 and 93]{ku}), however they do not hold for general pseudorings.

\begin{proposition}\label{lg1}
Let $R$ be a pseudomeadow and $I$ an ideal of $R$. Suppose that $I_\m=(0)$ for every $\m\in\mathrm{MaxSpec}(R)$. Then $I=(0)$.
\end{proposition}

\begin{proof}
Suppose to the contrary. Let $a\in I\setminus \{0\}$. By Theorem \ref{charact_pseudomeadow} and Proposition \ref{charact_red} there exists a maximal ideal $\m$ not containing $a$. 
From $a/s=0$ $(s\in R\setminus \m)$ we get $as'=0$ for some $s'\in R\setminus \m$. However, $as'\in R\setminus \m$, a contradiction. Thus $I=(0)$.
\end{proof}

\begin{remark}
Note that an analogous statement is not true for general pseudorings. For example, in the pseudoring $\Z_6$ with the all-products-zero meultiplication, the ideal $\m=\{0,2,4\}$ is maximal, however the complement $R\setminus \m$ is not multiplicatively closed, so we even cannot localize at $\m$.
\end{remark}

\begin{proposition}\label{lg2}
Let $R$ be a pseudomeadow and $I,J$ two ideals of $R$. Suppose that $I_\m=J_\m$ for every $\m\in\mathrm{MaxSpec}(R)$. Then $I=J$.
\end{proposition}

\begin{proof}
Consider the pseudomeadow $R/I$. Its maximal ideals are the ideals $\m/I$, where $\m$ is a maximal ideal of $R$ which contains $I$. Let $f:(R/I)_{\m/I}\to R_\m/I_\m$ be the canonical isomorphism. Under that isomorphism, $\displaystyle{
f(((I+J)/I)_{\m/I})
=\frac   {(I+J)_\m}  {I_\m}
=\frac{I_\m+J_\m}{I_\m}
=\frac{I_\m}{I_\m}=(0)
}$. 
Hence $((I+J)/I)_{\m/I}=(0)$. Since this holds for all maximal ideals of $R/I$, we have, by the previous proposition, $(I+J)/I=(0)$. Hence $J\subseteq I$. By symmetry, $I\subseteq J$. Thus $I=J$. 
\end{proof}


\bigskip
\small
\begin{thebibliography}{99}

\bibitem{bb}
J.~A.~BERGSTRA, I.~BETHKE,
     \textit{Subvarities of the Variety of Meadows},
     Sci. Ann. Comput. Sci., {\bf 27}(2017), 1-18.
\bibitem{brs}
I.~BETHKE, P.~RODENBURG, A.~SEVENSTER,
     \textit{The structure of finite meadows},
     J. Log. Algebr. Methods Program. {\bf 84}(2015), 276-282.
\bibitem{b}
G.~BIRKHOFF,
     \textit{Subdirect unions of universal algebras},
     Bull. Am. Math. Soc. {\bf 50}(1944), 764-768.
\bibitem{bo}
N.~BOURBAKI,
     \textit{Commutative Algebra},
     Hermann, Paris and Addison-Wesley, Reading, MA, 1972. 
\bibitem{g}
K.~R.~GOODEARL,
    \textit{Von Neumann Regular Rings},
    Pitman, London San Francisco Melbourne, 1979.
\bibitem{k}
I.~KAPLANSKY,  
     \textit{Commutative Rings},
     Revised Edition, The University of Chicago Press, Chicago and London, 1974. 
\bibitem{ko}
G.~K\"OTHE,
     \textit{Abstrakte Theorie nichkommutativer Ringe mit einer Anwendung auf die Darstellungstheorie kontinuierlicher Gruppen},
     Math. Ann., {\bf 103}(1930), 545-572.
\bibitem{kul}
H.~KULOSMAN,
     review MR3310420 for Mathematical Reviews for the paper {\it ``The structure of finite meadows"} by I.~Bethke, P.~Rodenburg, A.~Sevenster, J. Log. Algebr. Methods Program., {\bf 84}(2015), 276-282.  
\bibitem{ku}
E.~KUNZ,
      \textit{Introduction to Commutative Algebra and Algebraic Geometry},
      Birkh\"auser, Boston, 1985.
\bibitem{m}
N.~H.~MCCOY,
     \textit{Subrings of infinite direct sums},
     Duke Math. J. {\bf 4}(1938), 486-494.   
\bibitem{o}
J.-P.~OLIVIER,
     \textit{Anneaux absolument plats universels et \'epimorphismes \`a buts r\'eduits},
     S\'eminaire Samuel. Alg\`ebre commutative, {\bf 2}(1967/68), exp. n$^\circ$ 6, 1-12.
\bibitem{n}
J.~VON NEUMANN,
     \textit{On regular rings},
     Proc. Nat. Acad. Sci., {\bf 22}(1936), 707-713.

\end{thebibliography}
\end{document}